\documentclass[11pt]{article}
\usepackage{amsmath}
\usepackage{amssymb}
\usepackage{color}

\newcommand\eq[1]{(\ref{eq:#1})}

\newcounter{remark}[section]

\def\claim{\par\medskip\noindent\refstepcounter{remark}\hbox{\bf Remark \arabic{section}.\arabic{remark}}
	\ 
}
\def\endclaim{
	\par\medskip}
\newenvironment{remark}{\claim}{\endclaim}
\newtheorem{theorem}{Theorem}
\newcommand{\ti}{\mathbf{t}}

\newcommand{\SSS}{\mathbf{S}}

\newtheorem{definition}{Definition}

\newtheorem{corollary}{Corollary}
\def\endpf{{\ \hfill\hbox{\vrule width1.0ex height1.0ex}\parfillskip 0pt
	}}
	\newenvironment{proof}{\noindent{\bf Proof:}}{\endpf}
\begin{document}
		
\title{Asymptotic Independence of Regenerative Processes with dependent cycles}
\author{Royi Jacobovic\thanks{Department of Statistics; The Hebrew University of Jerusalem; Jerusalem 9190501; Israel.
{\tt royi.jacobovic@mail.huji.ac.il/offer.kella@gmail.com}}\and Offer Kella\footnotemark[1] \thanks{Supported by grant No. 1647/17 from the Israel Science Foundation and
the Vigevani Chair in Statistics.}}

\date{November 21, 2017}
		\maketitle
		\begin{abstract}
 We identify general conditions under which regenerative processes with dependent cycles and cycle lengths are asymptotically independent. The result is applied to various models. In particular, independent L\'evy processes with dependent secondary jumps at the origin (e.g., workloads of parallel M/G/1 queues with server vacations), the asymptotic performance of real-time status systems with multiple correlated sources measured by the stationary probability of an updated system and asymptotic results for clearing processes with dependent arrivals of inputs and clearings.
		\end{abstract}


\bigskip
\noindent {\bf AMS Subject Classification (MSC2010):} Primary 60K05; Secondary 	60K25, 60G51, 	90B15.

		\section{Introduction} We start by describing four models in order to motivate the need for the main result of this paper. All four models have a common structure. Namely, they are all dependent regenerative processes which are constructed from an i.i.d. sequence of dependent cycles and cycle lengths. That is, the structure considered more formally in Section~\ref{sec:main-result}, where we will give the main result (Theorem~\ref{thm:main result}) and its proof. In Section~\ref{sec:applications} we will see how to apply it to each of the four models considered below. For background on (standard) regenerative processes see, e.g., Chapters VI and VII of \cite{a03}.
		
		\subsection{L\'evy driven queues with secondary jump inputs}\label{sec:queues}
Consider $m\ge 2$ independent L\'evy processes (right continuous processes having stationary and independent increments) with no negative jumps, each starting from zero, denoted by $X_1(\cdot),\ldots,X_m(\cdot)$  satisfying $EX^i(1)<0$. E.g., the net input processes in an M/G/1 queue with traffic intensity less than one. For the $i$th process we let $\{U^i_n|\,n\ge 0\}$ be a sequence of positive random variables. The $i$th queue starts from $U^i_1$. When it first approaches zero in jumps to $U^i_2$ and continues until it approaches zero again, then jumps to $U^i_3$ and so on. Formally we can define $S_0=0$ and for $n\ge 1$
\begin{align}
S_n^i&=\inf\left\{t\left|\,X^i(t)+\sum_{j=1}^nU^i_j=0\right.\right\}\,,\nonumber\\
N^i(t)&=\sup\{n|\,S_n^i\le t\}\,,\\
X_i(t)&=X^i(t)+\sum_{i=1}^{N^i(t)+1}U^i_j\,.\nonumber
\end{align}
This type of model has been considered in the literature (e.g., \cite{cw93,kw90,kw91,kw92}) some of which was motivated by workload in a queue with server vacations. Assume now that $\{(U^1_n,\ldots,U^m_n)|\,n\ge 1\}$ is a sequence of i.i.d. random vectors, but with $U^1_1,\ldots,U^m_1$ being possibly dependent. In this case the content processes $X_1(t),\ldots,X_m(t)$ are possibly dependent for each $t\ge 0$. Does the joint content process have a limiting distribution? If yes, then under what conditions and what it is? Obviously, if $U^1_1,\ldots,U^m_1$ are independent then the content processes are independent and the answer reduces to the one dimensional case which can be found, for example, in \cite{kw91,kw92}.

\subsection{L\'evy driven clearing processes}\label{sec:clearing}
Consider $m\ge 2$ independent subordinators (nondecreasing L\'evy processes) denoted by $J_1(\cdot),\ldots,J_m(\cdot)$ (starting from zero).
For each $i$ let $\{S^i_n|\,n\ge 0\}$ be a strictly increasing sequence of random variables with $S^i_0=0$. These will be the {\em clearing times}. The $i$th process then behaves like a subordinator with the exception of the times $S^i_n$ where it is restarted from zero. Formally, if we let $N^i(t)=\sup\{n|\ S_n^i\le t\}$, then the $i$th process is the amount of content that accumulated since the most recent clearing, that is,
\begin{equation}
X_i(t)=J_i\left(t-S^i_{N^i(t)}\right)\ .
\end{equation}
Now, for $n\ge 1$ denote $T_n^i=S_n^i-S_n^{i-1}$ and assume that $$\{(T^1_n,\ldots,T^m_n)|\,n\ge 1\}$$ are i.i.d. random vectors with $T^1_1,\ldots,T^m_1$ being possibly dependent. Then precisely the same questions that were asked in Subsection~\ref{sec:queues} may be repeated here with respect to the content processes considered in this subsection. Clearing processes have been extensively studied (L\'evy driven or otherwise). See, e.g., \cite{w81,s74,s77,ss78,k98,ks15}.

\subsection{Real-time status updating}\label{sec:updating}
		This model is about, so called, real-time status updating. The background for this model is the fact that nowadays we witness an extensive use of variety of portable electronic devices. At any time, these devices update their users about the most recent information including social media, weather conditions, traffic congestions, prices of financial assets, etc. The point is that due to limited resources of the communication network, these status updates have to be loaded into the system with respect to some updating policy in order to maintain some adequate level of information freshness. As mentioned by \cite{Huang2015, yates2012}, one class of real-time status updating problems is concerned with the case where updates are generated by multiple sources. Here, the generation processes of updates by $m\ge 2$ different sources are assumed to be a $m$-dimensional renewal process with inter-arrival times which are distributed like $(T^1,\ldots,T^m)$.

Now, for each $1\le 1\le m$ and $n\geq1$ let $T_n^i$ be the time between the appearances of the $(n-1)$st and $n$th updates which were generated by the $i$th source. In addition, the sizes (in bytes) of updates generated by the $j$th source  $\{Y_j^i;j\geq 1\}$ are i.i.d positive random variables. It is assumed that $$(T_n^1,\ldots,T_n^m),Y^i_j;\,n\ge 0,\,j\ge 1,\,1\le i\le m$$ are all independent.

Moreover, for each source $i$ there exists a private channel with capacity $c_i$ bytes per second which operates according to a last-in-first-out policy with no waiting room. That is, if the channelling of an update is interrupted by a new update, then the old update is discarded from the system with no retrieval option. Notice that an update has been received by the system at a certain time if only if all of its content has already gone through the channel into the system until that time. Moreover, we say that the system is updated  at a certain time if and only if for every source, the most recent update has already been received by the system. The performance measure of this system is defined as the limiting probability that the system is updated. Of course, we need to identify conditions under which such a limit distribution exists and possibly compute it.

\subsection{Jackson networks}\label{subsec:Jackson}
Assume that $X_i(t)$ is the number of customers in the $i$th station in an open Jackson network with $m\ge 2$ stations. It is well known that under some standard conditions, the joint limiting distribution of $(X_1(t),\ldots,X_m(t))$ is of product form. This may be found in virtually any textbook on queueing theory. In this case, for positive $\alpha_i\not=\alpha_j$ for $i\not=j$ and arbitrary $\beta_1,\ldots,\beta_m$ does $(X_1(\alpha_1 t+\beta_1),\ldots,X_m(\alpha_m t+\beta_m))$ also have a limiting distribution? If yes, then is it also of product form?

		\section{The main result}\label{sec:main-result}
		For $m\ge 2$ denote $M_m=\{1,\ldots,m\}$ and let
\begin{equation}\label{eq:1n}
\{(X_n^1(\cdot),T_n^1),\ldots,(X_n^m(\cdot),T_n^m)|\,n\ge 1\}
\end{equation}
be a sequence of i.i.d. random elements distributed like and independent of
\begin{equation}\label{eq:10}
\{(X^1(\cdot),T^1),\ldots,(X^m(\cdot),T^m)|\,n\ge 1\}\,,
\end{equation}
the latter having an arbitrary joint distribution.
For each $i\in M_m$ we assume that
		\begin{description}
			\item{(i)} $T^i$ is almost surely (a.s.) nonnegative, has a nonarithmetic distribution with positive and finite mean denoted by $\mu_i$.
			\item{(ii)} $X^i(\cdot)$ takes values in some metric space and a.s. has right continuous sample paths.
		\end{description}
For each $i\in M_m$ denote $S_0^i=0$, $S_n^i=\sum_{k=1}^n T^i_n$ for $n\ge 1$, $N^i(t)=\sup\{n|S_n^i\le t\}$ and finally
\begin{equation}\label{eq:reg}
X_i(t)=X^i_{{N^i(t)+1}}\left(t-S^i_{N^i(t)}\right)
\end{equation}
and we observe that $X_i(\cdot)$ is a regenerative process which, due to (i) and (ii), converges in distribution.

The following is our main result. For its proof, it is practical to denote $\ti=(t_1,\ldots,t_m)$, $\mathbf{v}(t)=(v_1(t),\ldots,v_m(t))$, $\SSS_n=(S^1_n,\ldots,S^m_n)$, $\mathbf{T}=(T^1,\ldots,T^m)$ and $K=\prod_{i=1}^m\|f_i\|_\infty$.
		
		\begin{theorem}\label{thm:main result}
			Assume that $\mu_1\le\ldots\le\mu_m$ and let $v_1(\cdot),\ldots,v_m(\cdot)$ be non-negative (deterministic Borel) functions such that
			$v_m(t)\rightarrow\infty$ as $t\to\infty$ and for each $1\le i<m$,
				\begin{equation*}
				\liminf_{t\to\infty}\frac{v_i(t)}{v_{i+1}(t)}>\frac{\mu_i}{\mu_{i+1}}\,.
				\end{equation*}
	
			Then, $\left(X_1\left(v_1(t)\right),\ldots,X_m\left(v_m(t)\right)\right)$ converges (jointly) in distribution, as $t\to\infty$, to $\left(X_1(\infty),\ldots,X_m(\infty)\right)$ where $X_1(\infty),\ldots,X_m(\infty)$ are independent random variables, where for each $i$ and nonnegative Borel $g$ we necessarily have that
\begin{equation}
Eg(X_i(\infty))=\frac{1}{ET^i}E\int_0^{T^i}g(X^i(s))ds\,.
\end{equation}
		\end{theorem}

	\begin{proof}
		Let $f_1,\ldots,f_m$ be bounded and continuous functions. It suffices to show that
		\begin{equation}\label{eq:AsympInd}
		\exists\lim_{t\to\infty}E\prod_{i=1}^mf_i\left(X_i\left(v_i(t)\right)\right)=\prod_{i=1}^mEf_i\left(X_i(\infty)\right)\,.
		\end{equation}
		This will be done by induction. We assume without loss of generality that $\mu_1=\ldots=\mu_m=1$. Otherwise we can set $\tilde X^n_i(t)=X^n_i(\mu_i t)$, $\tilde T_n^i=T_n^i/\mu_i$ and $\tilde v_i(t)=v_i(t)/\mu_i$.

As mentioned below \eq{reg}, \eq{AsympInd} holds for $m=1$. When $m\ge 2$, assume that \eq{AsympInd} holds for $m-1$. For each $J\subset M_m$ and $\ti\ge 0$ denote the event
 \begin{equation}
 C_J(\ti)=\left(\bigcap_{i\in J}\{T_1^i>t_i\}\right)\cap\left(\bigcap_{i\not\in J}\{T_1^i\le t_i\}\right)
 \end{equation}
and observe that $\{C_J(\ti)|\,J\subset M_m\}$ are pairwise disjoint and their union is the entire sample space.	 For each $m\ge 1$ and $\ti\ge 0$ let
\begin{equation}
A_m(\ti)=E\prod_{i=1}^mf_i\left(X_i(t_i)\right)
\end{equation}
and for each $J\subset M_m$
\begin{equation}
a_J(\ti)=E\left(\prod_{i=1}^mf_i\left(X_i(t_i)\right)\right)1_{C_J(\ti)}\,.
\end{equation}
Clearly,
\begin{equation}
A_m(\ti)=\sum_{J\subset M_m}a_J(\ti)\ .
\end{equation}
By a standard renewal argument
\begin{equation}
a_{\phi}(\ti)=EA\left(\ti-\mathbf{T}\right)1_{\{\mathbf{T}\le\ti\}}
\end{equation}
thus, we obtain a multivariate renewal equation for $A_m(\cdot)$ of the form
\begin{equation}
A_m(\ti)=\sum_{\phi\not=J\subset M_m}a_J(\ti)
+EA\left(\ti-\mathbf{T}\right)1_{\{\mathbf{T}\le\ti\}}
\end{equation}
and by identical arguments as for the univariate case its unique solution is
\begin{align}
A_m(\ti)&=\sum_{n=0}^\infty\left(\sum_{\phi\not=J\subset M_m} a_J(\ti-\SSS_n)\right)1_{\{\SSS_n\le \ti\}}\nonumber\\
&=\sum_{\phi\not=J\subset M_m} \left(\sum_{n=0}^\infty a_J(\ti-\SSS_n)1_{\{\SSS_n\le \ti\}}\right)\ ,
\end{align}
in particular if we replace $\ti$ by $\mathbf{v}(t)$.

First we will show that for all $J\not=\phi,\{m\}$,
\begin{equation}\label{eq:vanish}
\lim_{t\to\infty}\sum_{n=0}^\infty a_J(\mathbf{v}(t)-\SSS_n)1_{\{\SSS_n\le\mathbf{v}(t)\}}= 0\,,
\end{equation}
after which it will be left to consider only $J=\{m\}$ where the induction step will be carried out.

Observe that for each $J\not=\phi,\{m\}$ either $i,j\in J$ for some $i<j$ or $i\in J$ for exactly one $i<m$. In the former case we have that, for all $\ti\ge0$,
\begin{equation}
a_J(\ti)\le KP(T^i>t_i,T^j>t_j)\ .
\end{equation}
This together with the assumptions between \eq{1n} and \eq{10} imply that, for each $n\ge 0$,
\begin{align}\label{eq:bound1}
Ea_J(\ti-\SSS_n)1_{\{\SSS_n\le \ti\}}&\le KP(T^i>t_i-S_n^i,T^j>t_j-S_n^j,S_n^i\le t_i,S_n^j\le t_j)\nonumber\\
&=KP(T^i_{n+1}>t_i-S_n^i,T^j_{n+1}>t_j-S_n^j,S_n^i\le t_i,S_n^j\le t_j)\nonumber\\
&=KP(S_n^i\le t_i<S_{n+1}^i,S_n^j\le t_j<S_{n+1}^j)\\
&=KP(N^i(t_1)=N^j(t_j)=n)\nonumber
\end{align}
and upon summing with respect to $n\ge 0$ we have that
\begin{equation}\label{eq:bound2}
\sum_{n=0}^\infty a_J(\ti-\SSS_n)1_{\{\SSS_n\le \ti\}}\le KP(N^i(t_i)=N^j(t_j))\ .
\end{equation}
For the case where $i\in J$ for exactly one $i<m$ we have that for all $\ti\ge 0$,
\begin{equation}
a_J(\ti)\le KP(T^i>t_i,T^m\le t_m)\,,
\end{equation}
and repeating the ideas in \eq{bound1} and \eq{bound2} gives that for this case
\begin{equation}
\sum_{n=0}^\infty a_J(\ti-\SSS_n)1_{\{\SSS_n\le \ti\}}\le KP(N^i(t_i)<N^m(t_j))\ .
\end{equation}
Thus, it follows that for each $J\not=\phi,\{m\}$, for some $i<j$ we have that
\begin{equation}
\sum_{n=0}^\infty a_J(\mathbf{v}(t)-\SSS_n)1_{\{\SSS_n\le \ti\}}\le KP(N^i(v_i(t))\le N^j(v_j(t)))\ .
\end{equation}
Now,
\begin{equation}\label{eq:gt0}
\frac{N^i(v_i(t))-N^j(v_j(t))}{v_j(t)}=\frac{v_i(t)}{v_j(t)}\frac{N^i(v_i(t))}{v_i(t)}-\frac{N^j(v_j(t))}{v_j(t)}
\end{equation}
where, since
\[
\liminf_{t\to\infty} \frac{v_i(t)}{v_j(t)}\ge \prod_{k=i}^{j-1}\liminf_{t\to\infty}\frac{v_k(t)}{v_{k+1}(t)}>1\,
\]
the $\liminf$ of the right hand side of \eq{gt0} is strictly positive. This implies that $N^i(v_i(t))-N^j(v_j(t))\to\infty$ as $t\to\infty$ a.s. and thus $P(N^i(v_i(t))\le N^j(v_j(t)))$ vanishes as $t\to\infty$, which in turn implies \eq{vanish} for all $J\not=\phi,\{m\}$.

It therefore remains to look at $a_{\{m\}}(\ti)$. For this case regenerative arguments lead to
\begin{align}
a_{\{m\}}(\ti)&=EA_{m-1}(t_1-T^1,\ldots,t_{m-1}-T^{m-1})f_m(X^m(t_m))1_{C_{\{m\}}}(\ti)\nonumber \\
&=EA_{m-1}(t_1-T^1,\ldots,t_{m-1}-T^{m-1})1_{\{T^1\le t_1,\ldots,T^{m-1}\le t_{m-1}\}}\\
&\qquad \cdot f_m(X^m(t_m))1_{\{T^m>t_m\}}\ .
\end{align}
and once again, the same ideas as for \eq{bound1}, that is, in this case replacing $(T^1,\ldots,T^m,X^m)$ by $(T^1_{n+1},\ldots,T^m_{n+1},X^m_{n+1})$, lead to
\begin{align}
&Ea_{\{m\}}(\ti-\SSS_n)1_{\{\SSS_n\le \ti\}}\\
&=EA_{m-1}\left(t_1-S^1_{n+1},\ldots,t_{m-1}-S^{m-1}_{n+1}\right)1_{\left\{S_{n+1}^1\le t_1,\ldots,S_{n+1}^{m-1}\le t_{m-1}\right\}}\nonumber\\
&\qquad \cdot f\left(X_{n+1}^m\left(t_m-S^m_n\right)\right)1_{\left\{S^m_n\le t_m<S_{n+1}^m\right\}}\nonumber\\
&=EA_{m-1}\left(t_1-S^1_{N^m(t_m)+1},\ldots,t_{m-1}-S^{m-1}_{N^m(t_m)+1}\right)\nonumber\\
&\qquad \cdot 1_{\left\{S_{N^m(t_m)+1}^1\le t_1,\ldots,S_{N^m(t_m)+1}^{m-1}\le t_{m-1}\right\}} f(X_m(t_m))1_{\{N^m(t_m)=n\}}\ .\nonumber
\end{align}
Summing over $n\ge 0$ gives
\begin{align}
&\sum_{n=0}^\infty Ea_{\{m\}}(\ti-\SSS_n)1_{\{\SSS_n\le \ti\}}\\
&=EA_{m-1}\left(t_1-S^1_{N^m(t_m)+1},\ldots,t_{m-1}-S^{m-1}_{N^m(t_m)+1}\right)\nonumber\\
&\qquad \cdot 1_{\left\{S_{N^m(t_m)+1}^1\le t_1,\ldots,S_{N^m(t_m)+1}^{m-1}\le t_{m-1}\right\}} f(X_m(t_m))\ .\nonumber
\end{align}
Next, we need to substitute $t_i=v_i(t)$ and verify the induction step.
If we show that
\begin{align}\label{eq:Am1}
&A_{m-1}\left(v_1(t)-S^1_{N^m(v_m(t))+1},\ldots,v_{m-1}(t)-S^{m-1}_{N^m(v_m(t))+1}\right)\\
&\qquad \cdot 1_{\left\{S_{N^m(v_m(t))+1}^1\le v_1(t),\ldots,S_{N^m(t_m)+1}^{m-1}\le v_{m-1}(t)\right\}}\ \nonumber
\end{align}
converges a.s. to $\prod_{i=1}^{m-1}Ef_i(X_i(\infty))$, then since
$X_m(v_m(t))$ converges in distribution to $X_m(\infty)$ as $t\to\infty$, the result will follow from Slutsky's Theorem.

For $i<m$ we first observe that
\begin{align}
I^i(t)\equiv\frac{S^i_{N^m(v_m(t))+1}}{v_m(t)}=\frac{N^m(v_m(t))+1}{v_m(t)}\frac{S_{N^m(v_m(t))+1}^i}{N^m(v_m(t))+1}\stackrel{\text{a.s.}}{\to}1\,. \end{align}
Thus,
\begin{equation}
\liminf_{t\to\infty}\frac{v_i(t)-v_m(t)I^i(t)}{v_m(t)}=\liminf_{t\to\infty}\frac{v_i(t)}{v_m(t)}-\lim_{t\to\infty}I^i(t)>1-1=0\ .
\end{equation}
This in turn implies that $v_i(t)-S_{N^m(v_m(t))+1}^i=v_i(t)-v_m(t)I^i(t)\stackrel{\text{a.s.}}{\to}\infty$, hence,
\begin{equation}
1_{\{S_{N^m(v_m(t))+1}^1\le v_1(t),\ldots,S_{N^m(t_m)+1}^{m-1}\le v_{m-1}(t)\}}\stackrel{\text{a.s.}}{\to}1\ .
\end{equation}
It remains to show that for $1\le i<m-2$ we have that a.s.
\begin{equation}
\liminf_{t\to\infty}\frac{v_i(t)-v_m(t)I^i(t)}{v_{i+1}(t)-v_m(t)I^{i+1}(t)}>1\
\end{equation}
in which case we can apply the induction hypothesis and thus complete the proof. Equivalently it would suffice to show that
\begin{equation}
\liminf_{t\to\infty}\frac{\frac{v_i(t)}{v_{i+1}(t)}-1-\frac{v_m(t)}{v_{i+1}(t)}(I^i(t)-I^{i+1}(t))}{1-\frac{v_m(t)}{v_{i+1}(t)}I^{i+1}(t)}>0\ .
\end{equation}
Since $\limsup_{t\to\infty}(v_m(t)/v_{i+1}(t))<1$ it follows that $\frac{v_m(t)}{v_{i+1}(t)}(I^i(t)-I^{i+1}(t))$ vanishes a.s. Clearly we also have that  $\liminf_{t\to\infty}(v_m(t)/v_{i+1}(t))<1$ so that finally
\begin{equation}
\liminf_{t\to\infty}\frac{\frac{v_i(t)}{v_{i+1}(t)}-1}{1-\frac{v_m(t)}{v_{i+1}(t)}I^{i+1}(t)}\ge
\frac{\liminf_{t\to\infty}\frac{v_i(t)}{v_{i+1}(t)}-1}{1-\liminf_{t\to\infty}\frac{v_m(t)}{v_{i+1}(t)}}>0
\end{equation}
and the proof is complete.
	\end{proof}
	
	The following corollaries are immediate consequences of Theorem~\ref{thm:main result}.
	\begin{corollary}\label{cor:multipicate}
		If $X(\cdot)$ is a regenerative process whose regeneration time has a non-arithmetic distribution with finite and positive mean, then for every positive $\alpha_i\not=\alpha_j$ for $i\not=j$ and $\beta_1,\ldots,\beta_m\in\mathbb{R}$, \[X(\alpha_1t+\beta_1),\ldots,X(\alpha_mt+\beta_m)\] are asymptotically independent and identically distributed as $t\to\infty$.
	\end{corollary}
	\begin{corollary}\label{cor:shifts}
		If $\mu_i\not=\mu_j$ for all $i\not=j$, then for every $\beta_1,\ldots,\beta_m\in\mathbb{R}$, \[X_1(t+\beta_1),\ldots,X_m(t+\beta_m)\] are asymptotically independent as $t\to\infty$.
	\end{corollary}
	
Finally we remark that one important special case is
\begin{equation}
X_i(t)=(\beta_i(t),\gamma_i(t))\equiv\left(t-S^i_{N^i(t)},S^i_{N^i(t)+1}-t\right)\,,
\end{equation}
that is, the joint age and residual lifetime process associated with $N^i(\cdot)$, where $\{(T^1_n,\ldots,T^m_n)|n\ge 1\}$ are i.i.d. (with an arbitrary joint distribution). Denoting $(\beta_i,\gamma_i)$ a random vector having the limiting distribution of $(\beta_i(t),\gamma_i(t))$, provided of course that $0<\mu_i<\infty$ and that $T^i$ has a nonarithmetic distribution, and also denote $\alpha_i=\beta_i+\gamma_i$ and let $U_i\sim\text{Uniform}(0,1)$ be an independent random variable, then it is well known (and will be needed a bit later) that
\begin{align}
P(\beta_i\in dx)=P(\gamma_i\in dx)=F_e^i(dx)&\equiv\mu_i^{-1}P(T^i>x)dx\nonumber\\
P(\alpha_i\in dx)&=\mu_i^{-1}xP(T^i\in dx)\\
P(\beta_i>x,\gamma_i>y)&=P(\gamma_i>x+y)\nonumber\\
(\beta_i,\gamma_i)&\sim (U_i\alpha_i,(1-U_i)\alpha_i)\nonumber
\end{align}

	\section{The four models revisited}\label{sec:applications}
Let us see how to apply Theorem~\ref{thm:main result} to each of the three models considered in the Introduction.

\subsection{L\'evy driven queues with secondary jump inputs}
Recalling the notation in Subsection~\ref{sec:queues} and denoting $T^i_n=S^i_n-S^i_{n-1}$, we can directly apply the result provided that we know that $ET^i_1$ are all finite, are different from one another and moreover have a nonarithmetic distribution. In that case the limiting distribution exists and is of product form with marginals given in, e.g., \cite{kw91,kw92}. The only way for $T^1_1$ to have an arithmetic distribution is if for some $d>0$, $X_i(\cdot)$ is a compound Poisson process with negative drift and jump distribution concentrated on $\{nd|\,n\ge 0\}$ and the distribution of $U^i_1$ is also concentrate on $\{nd|\,n\ge 0\}$. In all other cases the distribution of $T^1_i$ is nonarithemtic. In particular, when (but not only when) the distribution of $U^i_1$ has a nonarithmetic distribution or when $X_i(\cdot)$ is either not a compound Poisson process with drift or if it is and the jump sizes have a nonarithmetic distribution. It is well known (e.g., \cite{kw91,kw92}) that $ET^i_1=(-X^i_1(1))EU^i_1$, thus in order to apply our result we need that $EX_i(1)>-\infty$ (it can never be $+\infty$), that $EU^i_1<\infty$ and that the quantities $(-X^i_1(1))EU^i_1$ are different from one another.

\subsection{L\'evy driven clearing processes}
Recalling the notation in Subsection~\ref{sec:clearing} we only need that $T^i_1$ have nonarithmetic distributions as well as finite and different means. Then the limiting distribution is of product form with well known marginals (e.g., \cite{k98}).

\subsection{Real-time status updating}
	To start with, the following provides an exact definition of an updated system:
	
	\begin{definition}
		The system is updated at time $t>0$ if
		
		\begin{equation*}
		\beta_i(t)>Y^i_{N^i(t)}/c_i \ \ , \ \ \forall i=1,\ldots,m
		\end{equation*}
		where $\beta_i(\cdot)$ and $N^i(\cdot)$ are the age and counting processes which are associated with the generation process of updates by the $i$th source.
	\end{definition}
	Notice that since the purpose is identifying the system's asymptotic behaviour, it does not matter whether the system is assumed to be initially updated or not. Now, Using  Corollary \ref{cor:shifts} with respect to $X_i(\cdot)=\beta_i(\cdot)$ for every $1\le i\le m$, since $Y_j^i,(X_n^1,\ldots,X_n^m);\,n\ge 1,j\ge 1,1\le i\le m$ are all independent, it is straightforward that
	
	\begin{equation}\label{eq:pi}
	\pi=\prod_{i=1}^mE\bar{F}^i_e\bigg(\frac{Y^i}{c_i}\bigg)
	\end{equation}
	where $\bar{F}_e^i\equiv1-F_e^i,\forall i\in M_m$.
	
Finally, we note in passing that (\ref{eq:pi}) leads to a family of optimization problems of allocating positive capacities $c_i$ subject to various constraints. Although for each $i\in M_m$, $\bar F^i_e(\cdot)$ is convex (since $F^i_e(\cdot)$ has a nonincreasing density), which initially gives some hope, the product appearing in (\ref{eq:pi}) and the fact that $c_i$ appears in the denominator implies that as a function of $(c_1,\ldots,c_m)$ the right hand side is neither convex nor concave and thus in general the problem is not necessarily a simple one. As this is not the scope of this paper, we will not elaborate on this any further.

\subsection{Jacson networks}
Under the standard conditions, the joint process is regenerative, where the epochs when the network becomes empty are the regeneration epochs. These are known to have a finite mean and a distribution that has a density and is thus nonarithmetic.
Given Theorem~\ref{thm:main result} and in particular Corollary~\ref{cor:multipicate}, the answer to both questions raised in Subsection~\ref{subsec:Jackson} is clearly yes. In fact, due to the original product form result for Jackson networks, when $\beta_i=0$ for all $i\in M_m$, this is true even if we do not assume that the $\alpha_i$'s are different. When they {\em are} different, if we look at the entire ($m$-dimensional) network status at the time points $\alpha_it+\beta_i$ for $i=1,\ldots,k$ where $k$ is not necessarily equal to $m$, we will get a limit in distribution which is actually a product form of $mk$ distributions.

\begin{remark}
We conclude by admitting that there are clearly cases in which product form emerges when the assumptions of Theorem~\ref{thm:main result} do not hold. For example, product form in certain queueing networks, product form in certain reflected Brownian motions, the case where $X_1(\cdot),\ldots,X_m(\cdot)$ are independent, the simultaneous reduction model of \cite{mk01}, etc.
\end{remark}


\begin{thebibliography}{99}
\bibitem{a03} Asmussen, S. (2003). {Applied Probability and Queues: 2nd Edition}. Springer.
\bibitem{cw93} Chen, H. and W. Whitt. (1993). Diffusion Approximations for Open Queueing Networks with Service Interruptions. {\em Queueing Systems}, {\bf 13}, 335-359.
\bibitem{Huang2015} Huang, L. and E. Modiano. (2015). Optimizing age-of-information in a multi-class queueing system. {\em Proceeding of IEEE International Symposium on Information Theory (ISIT)}, 1681-1685.
\bibitem{k98} Kella, O. (1998). An exhaustive L\'evy storage process with intermittent output. {\em Stochastic Models}, {\bf 14}, 979-992.
\bibitem{mk01} Miyazawa, M. and O. Kella. (2001). Parallel fluid queues with constant inflows and simultaneous random reductions. {\em J. Appl. Probab.}, {\bf 38}, 609-620.
\bibitem{ks15} Kella, O. and W. Stadje. (2015). A clearing system with impatient passengers: asymptotics and estimation in a bus stop model. {\em Queueing Systems}, {\bf 80}, 1-14.
\bibitem{kw90} Kella, O. and W. Whitt. (1990). Diffusion approximations for queues with server vacations. {\em Adv. Appl. Probab.}, {\bf 22},706-729.
\bibitem{kw91} Kella, O. and W. Whitt. (1991). Queues with server vacations and \L'evy processes with secondary jump input. {\em Ann. Appl. Prob.}, {\bf 1}, 104-117.
\bibitem{kw92} Kella, O. and W. Whitt. (1992). Useful martingales for stochastic storage processes with L\'evy input. {\em J. Appl. Probab.}, {\bf 29}, 396-403.
\bibitem{ss78} Serfozo, R. F. and S. Stidham, Jr. (1978). Semistationary clearing processes. {\em Stoch. Proc. Appl.}, {\bf 6}, 165-178.
\bibitem{s74} Stidham, S., Jr. (1974). Stochastic clearing systems. {\em Stoch. Proc. Appl.}, {\bf 2}, 85-113.
\bibitem{s77} Stidham, S., Jr. (1977). Cost models for Stochastic clearing systems. {\em Oper. Res.}, {\bf 25}, 100-127.
\bibitem{w81} Whitt, W. (1981). The stationary distribution of a stochastic clearing process. {\em Oper. Res.}, {\bf 29}, 294-308.
\bibitem{yates2012} Yates, R. D. and S. Kaul. (2012). Real-time status updating: Multiple sources.  {\em Proceeding of IEEE International Symposium on Information Theory (ISIT)}, 2666-2670.
\end{thebibliography}
	\end{document}